\documentclass[11pt,a4paper,twoside]{article}
\usepackage[T1]{fontenc}
\usepackage[utf8]{inputenc}
\usepackage[english]{babel}
\usepackage{amssymb, amsmath, amsfonts, amsthm}
\usepackage{enumerate}
\usepackage{colonequals}
\usepackage{empheq,fancybox}
\usepackage{color}
\usepackage[small]{titlesec}
\usepackage[noblocks]{authblk}

\usepackage{microtype}
\usepackage{ellipsis}
\usepackage[BCOR=20mm,DIV=11]{typearea}
 \newcommand{\hm}[1]{\leavevmode{\marginpar{\tiny%
 $ \hbox to 0mm{\hspace*{-0.5mm} $ \leftarrow $ \hss}%
 \vcenter{\vrule depth 0.1mm height 0.1mm width \the\marginparwidth}%
 \hbox to
 0mm{\hss $ \rightarrow $ \hspace*{-0.5mm}} $ \\\relax\raggedright #1}}}

\newcommand{\euler}{\mathrm{e}}
\newcommand{\drm}{\mathrm{d}}

\newcommand{\RR}{\mathbb{R}}

\newcommand{\cC}{\mathcal{C}}

\newcommand{\tvert}[1]{{\left\vert\kern-0.25ex\left\vert\kern-0.25ex\left\vert #1
    \right\vert\kern-0.25ex\right\vert\kern-0.25ex\right\vert}}
\renewcommand{\epsilon}{\varepsilon}

\DeclareMathOperator{\diam}{\mathop{diam}}

\DeclareMathOperator{\Deg}{Deg}
\DeclareMathOperator{\arctanh}{arctanh}
\newtheorem{theorem}{Theorem}[section]

\newtheorem{proposition}[theorem]{Proposition}
\newtheorem{corollary}[theorem]{Corollary}
\theoremstyle{definition}
\newtheorem{definition}[theorem]{Definition}
\theoremstyle{remark}
\newtheorem{remark}[theorem]{Remark}
\newtheorem{example}[theorem]{Example}
\begin{document}
 \title{Distance bounds for graphs with some negative Bakry-\'Emery curvature}
 \author{Shiping Liu, Florentin M\"unch, Norbert Peyerimhoff, Christian Rose}
 \date{\today}
 \maketitle
 \begin{abstract}
  We prove distance bounds for graphs possessing positive Bakry-\'Emery curvature apart from an exceptional set, where the curvature is allowed to be non-positive. If the set of non-positively curved vertices is finite, then the graph admits an explicit upper bound for the diameter. Otherwise, the graph is a subset of the tubular neighborhood with an explicit radius around the non-positively curved vertices. Those results seem to be the first assuming non-constant Bakry-\'Emery curvature assumptions on graphs.
 \end{abstract}

 \section{Introduction}
 In Riemannian geometry, diameter bounds for complete connected Riemannian manifolds are well established under several curvature assumptions. The well known Bonnet-Myers theorem states that if the Ricci curvature of a manifold is larger than a positive threshold, the diameter of the manifold is finite, and, therefore, the manifold itself compact, see \cite{Jost-08,Petersen-16} and the references therein. In particular, the classical Jacobi field technique used there provides also a sharp upper estimate for the diameter. Later on, this result was generalized in \cite{PetersenSprouse-98}. There, the authors assumed that the amount of the Ricci curvature of the manifold $M$ below a positive level is locally uniformly $L^p$-small for some $p>\dim M/2$, and obtain indeed a diameter bound depending on this kind of smallness of the curvature.\\
The concept of Ricci curvature was transferred into various settings. Let us provide a brief summary of the history.
Already in 1985, Bakry and \'Emery introduced Ricci curvature on diffusion semigroups via the highly generalizable $\Gamma$-calculus \cite{bakry1985diffusions} derived from the Bochner formula.
This approach has first been applied to a discrete setting in \cite{elworthy91} and diversely used in \cite{schmuckenschlager1998curvature,bauer2015li,horn2014volume,munch2017remarks,munch2014li,lin2010ricci,fathi2015curvature,gong2015properties,hua2017stochastic,LiuMuenchP-16,chung2014harnack}.
The theory of local metric measure spaces has also benefitted from the Bakry-\'Emery approach. For more information about Ricci-curvature on metric measure spaces, see \cite{ambrosio2014metric,erbar2015equivalence,lott2009ricci,sturm2006geometry}.
A concept of Ricci curvature on graphs via optimal transport has been introduced by Ollivier \cite{ollivier2009ricci} and applied in \cite{lin2010ricci,lin2011ricci,bauer2011ollivier,jost2014ollivier}.
Recently, Erbar,  Maas and Mielke introduced a Ricci curvature on graphs via convexity of the entropy \cite{erbar2012ricci,fathi2016entropic,erbar2017ricci,mielke2013geodesic}.
In a highly celebrated paper, Erbar, Kuwada and Sturm proved that on metric measure spaces, the concepts of Ricci curvature via $\Gamma$-calculus (Bakry-\'Emery) and optimal transport and entropy (Lott-Sturm-Villani) coincide \cite{erbar2015equivalence}. On the other hand, in the setting of graphs, Bakry-{\'E}mery Ricci curvature and Ollivier Ricci curvature are often quite different and there are many open questions about the relations between these curvature notions.

It is now natural to ask for analogues and generalizations of the diameter bounds for manifolds above to contexts in which concepts of Ricci curvature exist.
For metric measure spaces, there have been attempts to generalize the Bonnet-Myers theorem to variable Ricci curvature bounds in an integral sense, see \cite{Ketterer-15} and the references therein. For connected graphs $G=(V,E)$, the authors of \cite{LiuMuenchP-16} show a sharp diameter bound assuming positive Bakry-\'Emery curvature in the $CD(K,N)$-setting for $N\in (0,\infty]$, notions of curvature we will introduce below. For convenience, we recall the result for further reference.
\begin{theorem}[\cite{LiuMuenchP-16}]\label{thm:LMP}
Let $G=(V,w,m)$ be a graph.
\begin{enumerate}
 \item
  Assume that $CD(K,\infty)$ holds for $K>0$ and the graph admits an upper bound $\Deg_{\max}$ for the weighted vertex degree.
Then, we have $$\diam_d(G)\leq \frac{2\Deg_{\max}}K,$$ where $\diam_d$ is the diameter of $G$ with respect to the combinatorial distance.
 \item

 Assume that $CD(K,N)$ holds for $K,N>0$.
   Suppose that $G$ is complete in the sense of \cite{hua2017stochastic} and satisfies $\inf_{x\in V} m(x) >0$.
 Then, we have $$\diam_\sigma(G)\leq \pi\sqrt{\frac{N}{K}},$$ where $\sigma$ is the resistance metric defined below.
 \end{enumerate}
\end{theorem}
In this article, we generalize the above discrete Bonnet-Myers theorem to the situation where the graph is positively curved except on a vertex set $V_0$, where the curvature is allowed to be non-positive. The main result below states that a graph is always covered by the tubular neighborhood around the negatively curved vertices of an explicit radius depending on local curvature dimension assumptions, which are given pointwise by the Bochner formula shown below. This description of the curvature involves the Laplacian of the space considered. The idea is to compare the different curvature values on the sets $V_0$ and $V\setminus V_0$ via the semigroups associated to different Laplacians. On one hand, we have a graph of constant positive curvature, the lower curvature bound of $V\setminus V_0$, and a graph of constant negative curvature, the lower curvature bound of $V_0$. Those lead to different Laplacians and therefore to different semigroups, which have to be controlled in a manner such that the diameter of the whole graph can be bounded above. After we introduced the neccessary framework and the main result in the section below, we show several preparatory estimates of the semigroup depending on the set of negatively curved vertices and refine the analysis of the techniques developed in \cite{LiuMuenchP-16}.
 \section{Setting and main result}
 Let $G=(V,w,m)$ be a weighted, connected, locally finite graph. That is, on the vertex set $V\not=\emptyset$, we introduce a symmetric map $$w\colon V\times V\to [0,\infty),\quad (x,y)\mapsto w_{xy}=w_{yx}$$ and $$m\colon V\to (0,\infty),\quad x\mapsto m_x.$$
 If $x,y$ are two vertices with $w_{xy}>0$, we say they are neighbors, or they are connected by an edge, and write $x\sim y$.  We say $G$ is locally finite if each vertex has finitely many neighbors.
 The maps $w$ and $m$ introduced above represent the edge measure and the vertex measure of $G$, respectively.

For any two vertices of a connected graph, there is a path connecting them. The graph distance $d\colon V\times V\to [0,\infty]$ is given by the number of edges in a shortest path between two vertices. The diameter $\diam(V')$ of a set $V'\subset V$ is the maximum graph distance between any two vertices in $V'$. By $T_r(V')$ we denote the tubular neighborhood of $V'\subset V$ of radius $r>0$. If $V'=\{x\}$ for some $x\in V$, then $B_r(x):=T_r(V')$, the ball around $x$ with radius $r>0$. As usual, the weighted degree of a vertex $x\in V$ is given by
$$\Deg(x)=\frac 1{m(x)}\sum_{y\in V}w_{xy}.$$
We say that $G$ has bounded vertex degree if there exists $\Deg_{\max} < \infty$ with
$\Deg(x) \le \Deg_{\max}$ for all vertices $x \in V$.
Denote by $\cC_c(V)$ the set of finitely supported functions on $V$ and $\Vert\cdot\Vert_\infty$ the maximum norm. The Laplacian on functions $f\in\cC_c(V)$ is defined by
 $$\Delta f(x)=\frac{1}{m(x)}\sum_{y\in V}w_{xy}(f(y)-f(x)),\quad x\in V.$$
 \begin{remark}
  If $m(x)=\sum_{y\in V}w_{xy}$ for any $x\in V$, the associated Laplacian is called the normalized Laplacian. If $m(x)=1$ for any $x\in V$, the Laplacian is called combinatorial or physical.
 \end{remark}
The definition of the Laplacian leads to the so-called carr\'e du champ operator $\Gamma$: for all $f,g\in \cC_c(V)$, $x \in V$:
\begin{eqnarray*}
\Gamma(f,g)(x)&=&\frac 12(\Delta(fg)-f\Delta g-g\Delta f)(x) \\
&=& \frac{1}{2m(x)} \sum_{y \in V} \omega_{xy} (f(y)-f(x))(g(y)-g(x)).
\end{eqnarray*}
For simplicity, we always write $\Gamma(f):=\Gamma(f,f)$. Iterating $\Gamma$, we can define another form $\Gamma_2$, which is given by $$\Gamma_2(f,g)(x)=\frac 12(\Delta\Gamma(f,g)-\Gamma(f,\Delta g)-\Gamma(g,\Delta f))(x), \quad x\in V, f,g\in \cC_c(V).$$ We abbreviate $\Gamma_2(f)=\Gamma_2(f,f)$.\\
As mentioned before, the graph distance is defined by shortest paths between two points. In contrast, we can define another kind of metric coming from the operator $\Gamma$.
 \begin{definition}[Intrinsic/resistance metric]
 Let $G=(V,w,m)$ be a graph.
 \begin{enumerate}[(i)]
 \item
	A metric $\rho$ on $V$ is called \emph{intrinsic} if for all $x \in V$,
	\begin{align}
	\left\|\Gamma \rho(x,\cdot) \right\|_\infty \leq 1.
	\end{align}
 \item For an intrinsic metric $\rho$, the \emph{jump size} of $\rho$ is given by $$R_\rho:=\sup\{\rho(x,y)\mid x\sim y\}.$$
 \item The \emph{resistance} metric $\sigma$ on $V$ is given by $$\sigma\colon V\times V\to[0.\infty],\quad (x,y)\mapsto\sup\{f(y)-f(x)\mid\Vert\Gamma f\Vert_\infty\leq 1\}.$$
\end{enumerate}
 \end{definition}

As in the case of the graph distance, if $r$ is an intrinsic or the resistance metric, we define $\diam_r(V')$ for a subset $V'\subset V$ to be the diameter of $V'$ with respect to $r$, and $T^r_R(V')$ denotes the tubular neighborhood of $V'$ of radius $R$ with respect to $r$, etc. Intrinsic metrics have already been used to solve various problems on graphs, see, e.g., \cite{BHK-13,BKW-15,Fol-11,Fol-14,GHM-12, HKMW-13, HK-14}.
\begin{example}
 A natural intrinsic metric on a graph was introduced in \cite[Definition 1.6.4]{Huang-11} (see also \cite[Example~2.9]{hua2017stochastic}):
 $$\rho(x,y)=\inf_{\gamma} \sum_{i=0}^{n-1}(\Deg(x_i)\vee\Deg(x_{i+1}))^{-1/2},\quad x,y\in V,$$
 where $\gamma$ is a path $x=x_0\sim x_1\sim\ldots\sim x_{n-1}\sim x_n=y$, and $a\vee b:=\max\{a,b\}$ for $a,b\in \mathbb{R}$.
\end{example}

\begin{remark}\label{remark:intrinsic}\begin{enumerate}[(i)]
\item All metrics smaller than an intrinsic metric are intrinsic, too. In general, the resistance metric $\sigma$ is not intrinsic, but is greater than all intrinsic metrics. 

\item
The properties of the resistance metric rely on the properties of the underlying Laplacian.
It is shown in Proposition~\ref{prop:DistanceAndDimensionBoundedDegree} that  if $\Deg(x)\leq \Deg_{\max}<\infty$ for all $x\in V$, we have that $\rho$ is intrinsic with
$$
\rho(x,y) := \sqrt{\frac 2 {\Deg_{\max}}} d(x,y).
$$

\item If $\rho$ is intrinsic and all $\rho$-balls are finite, then $G$ is complete \cite[Theorem 2.8]{hua2017stochastic}. For the reader's convenience, we recall that a graph $G=(V,w,m)$ is complete in the sense of \cite{hua2017stochastic} if there exists a nondecreasing sequence of finitely supported functions $\{\eta_k\}_{k=1}^\infty$ such that $\lim_{k\to \infty}\eta_k=1_V$ and $\Gamma(\eta_k)\leq 1/k$, where $1_V$ is the constant function $1$ on $V$.

\end{enumerate}
\end{remark}

The operator $\Gamma$ not only leads to a definition of a metric, but also to the curvature conditions in the sense of Bakry-\'{E}mery.
 \begin{definition} Let $K\in \RR$ and $N\in (0,\infty]$.
  \begin{enumerate}[(i)]
   \item Define the pointwise curvature dimension condition $CD(K,N,x)$ for $x\in V$ by $$\Gamma_2(f)(x)\geq K\Gamma(f)(x)+\frac 1N(\Delta f)^2(x),\,\,\,\text{for any }\,\,\,f: V\to \mathbb{R}.$$
   \item The curvature dimension condition $CD(K,N)$ holds iff $CD(K,N,x)$ holds for any $x\in V$.
   \item For any $x\in V$, we define $$K_{G,x}(N):=\sup\{K\in\RR\mid CD(K,N,x)\}.$$
  \end{enumerate}
 \end{definition}

We will need different assumptions to guarantee the semigroup characterization of Bakry-\'{E}mery curvature (see \cite{hua2017stochastic,gong2015properties}).
These assumptions are satisfied whenever the vertex measure $m$ is \emph{non-degenerate}, that is, $$\inf_{x\in V}m(x)>0,$$ and all balls with respect to an intrinsic metric are finite. In the case of bounded vertex degree $\Deg(x) \le \Deg_{\max}$ for
all $x \in V$, the non-degenerate vertex measure condition can usually be dropped.

In case of bounded vertex degree, the combinatorial distance is intrinsic up to a constant. Furthermore, we have a uniform control of the dimension in terms of the curvature.

\begin{proposition}\label{prop:DistanceAndDimensionBoundedDegree}
	Let $G=(V,w,m)$ be a graph with bounded degree $\Deg(x) \leq \Deg_{\max}$.
	Then, $\rho:=d \sqrt{\frac {2} {\Deg_{\max}}}$ is an intrinsic metric.
	Furthermore if $G$ satisfies $CD(K,\infty)$, it also satisfies $CD(K-s, \frac{2\Deg_{\max}}s)$ for all $s>0$.
\end{proposition}
\begin{proof}
	Let $x,x_0 \in V$ and let $g:=\rho(x_0,\cdot)$
	We have
	\begin{align*}
	\Gamma g(x) &= \frac 1 {2m(x)}\sum_{y} w(x,y) (g(x)-g(y))^2\\
	&\leq \frac 1 {2m(x)}\sum_{y} w(x,y) \sqrt{\frac {2} {\Deg_{\max}}}^2\\
	&\leq\frac {2} {\Deg_{\max}} \cdot \frac  {\Deg_{\max}}{2}\\
	&=1
	\end{align*}
	which shows that $\rho$ is an intrinsic metric.	
	Furthermore, $CD(K,\infty)$ implies for all $f$,
	\begin{align*}
	\Gamma_2 f \geq (K-s) \Gamma f + s \Gamma f \geq (K-s) \Gamma f + \frac s {2 \Deg_{\max}} (\Delta f)^2
	\end{align*}
	where the latter inequality follows from Cauchy-Schwarz. Hence, $G$ satisfies the condition
	$CD(K-s, \frac{2\Deg_{\max}}s)$ as claimed.
\end{proof}
The main theorem stated below extends Theorem \ref{thm:LMP} to the case of negatively curved vertices.
\begin{theorem}\label{thm:main}
  Let $G=(V,w,m)$ be a weighted, complete graph with non-degenerate vertex measure $m$, let $N\in(0,\infty]$, and $\rho$ an intrinsic metric on $G$. 
  Define $$V_0:=\{x\in V\mid K_{G,x}(N)\leq 0\}.$$
  \begin{enumerate}[(i)]
   \item If $N=\infty$, $V_0=\emptyset$, $\Deg(x)\leq \Deg_{\max}<\infty$ for any $x\in V$, and $CD(K_0,N)$ for $K_0>0$, then
   $$\diam_\rho(G)\leq \frac {2\sqrt{2\Deg_{\max}}}{K_0}.$$
   \item If $N<\infty$, $V_0=\emptyset$, and $CD(K_0,N)$ for $K_0>0$, then
   $$\diam_\rho(G)\leq \pi\sqrt{\frac N{K_0}}.$$
   \item If $N=\infty$, $V_0\not=\emptyset$, $\Deg(x)\leq \Deg_{\mathrm{max}}<\infty$ for all $x\in V$, and assuming, for $K,K_0>0$,
  \begin{align*}CD(-K_0,N,x)\quad\forall\, x\in V\text{ and }CD(K,N,x)\quad \forall \, x\in V\setminus V_0,
   \end{align*}
   then
   $$V\subset T_{R}(V_0),\quad R:= 1 + 18.2\sqrt{2} \, e^{4K_0/K}\,\frac{ \Deg_{\max}}{\sqrt{KK_0}}.$$
   \item If $N<\infty$, $V_0\not=\emptyset$, and assuming, for $K,K_0>0$,
  \begin{align*}CD(-K_0,N,x)\quad\forall\, x\in V\text{ and }CD(K,N,x)\quad \forall \, x\in V\setminus V_0,
   \end{align*}
 then $$V\subset T^\rho_{R}(V_0),\quad R:=R_\rho+18.2\euler^{2K_0/K}\sqrt{\frac{N}{K+K_0}}.$$
  \end{enumerate}
 \end{theorem}
 Note that (i) and (ii) in the above theorem are included in \cite{LiuMuenchP-16} since every intrinsic metric is dominated by the resistance metric and, therefore, $$\diam_\rho(G)\leq \diam_\sigma(G).$$ 
 We also point out that any locally finite graph with $\Deg_{\max}<\infty$ is complete by Proposition \ref{prop:DistanceAndDimensionBoundedDegree} and Remark \ref{remark:intrinsic} (iii).
 \section{CD conditions and semigroups}
 By the spectral calculus, we can associate to $\Delta$ the heat semigroup $(P_t)_{t\geq 0}$. Using a standard argument, we derive a commutation formula for the semigroup and the gradient depending on the set of negatively curved vertices.
 \begin{proposition}\label{gradientsg}
 Let $G=(V,w,m)$ be a weighted, complete graph with non-degenerate vertex measure $m$, and $N\in(0,\infty]$. Define $$V_0:=\{x\in V\mid K_{G,x}(N)\leq 0\}.$$  Let $K>0,K_0\geq 0$ such that $$CD(-K_0,N,x)\quad \forall \, x\in V\quad \text{and}\quad CD(K,N,x)\quad\forall \, x\in V\setminus V_0.$$
 Then for any bounded function $f\colon V\to \RR$ with bounded $\Gamma f$,
 \begin{align}\label{prop:sgc}
  \Gamma(P_Tf)(x)\leq &\euler^{-2KT}P_T(\Gamma f)(x)-\frac{1-e^{-2KT}}{KN}(\Delta P_Tf)^2(x)\notag\\
  &+2(K_0+K)\Vert\Gamma f\Vert_\infty\euler^{2K_0T}\int_0^T\euler^{-2(K+K_0)s}P_s1_{V_0}(x)\drm s.
 \end{align}
\end{proposition}
\begin{remark}
Let $G=(V,w,m)$ be a weighted, complete graph with non-degenerate vertex measure $m$. If $G$ satisfies $CD(K,N)$, $K\in \RR, N\in (0,\infty]$,
it was shown in \cite[Lemma 2.3]{LiuMuenchP-16} that for any bounded function $f\colon V\to \RR$ with bounded $\Gamma f$,
\begin{equation}\label{eq:LMP16}
\Gamma P_tf(x)\leq \euler^{-2Kt}P_t\Gamma f(x)-\frac{1-\euler^{-2Kt}}{KN}(\Delta P_tf)^2(x).
\end{equation}
So the estimate (\ref{prop:sgc}) is a refinement of (\ref{eq:LMP16}) in the setting of Proposition \ref{gradientsg}.
\end{remark}
\begin{proof}
 As in the classical Ledoux-ansatz, for a bounded function $f\colon V\to \RR$ and $T>0$, let
 $$F(s)=\euler^{-2Ks}P_s\Gamma P_{T-s}f(x)$$ and compute
 \begin{align*}
  F'(s)&=\euler^{-2Ks}\left(-2KP_s\Gamma P_{T-s}f(x)+\Delta P_s\Gamma P_{T-s}f(x)+2P_s\Gamma(P_{T-s}f,-\Delta P_{T-s}f)(x)\right)\\
  &=\euler^{-2Ks}P_s\left(-2K\Gamma P_{T-s}f+\Delta \Gamma P_{T-s}f-2\Gamma(P_{T-s}f,\Delta P_{T-s}f)\right)(x)\\
  &=\euler^{-2Ks}P_s\left(2\Gamma_2(P_{T-s}f)-2K\Gamma(P_{T-s}f)\right)(x).
  \end{align*}
  It is well known that the heat semigroup is generated by a smooth integral kernel, which is called the heat kernel. In particular, it can be proved that there is a pointwise minimal version, called $p: (0,\infty) \times V \times V \to \RR$, obtained via an exhaustion procedure by Dirichlet heat kernels on compact ($=$ finite) subsets of $V$ (see, e.g., \cite{linliu2015,aweber2010}).Therefore, we get
  \begin{align*}F'(s)
  &=2\euler^{-2Ks}\sum_{y\in V}p(s,x,y)\left[\Gamma_2(P_{T-s}f)(y)-K\Gamma(P_{T-s}f)(y)\right]\\
  &=2\euler^{-2Ks}\left[\sum_{y\in V\setminus V_0}p(s,x,y)\left[\Gamma_2(P_{T-s}f)(y)-K\Gamma(P_{T-s}f)(y)\right]\right.\\
  &\quad\left.+\sum_{y\in V_0}p(s,x,y)\left[\Gamma_2(P_{T-s}f)(y)+K_0\Gamma(P_{T-s}f)(y)\right]\right.\\
  &\quad\left.-(K+K_0)\sum_{y\in V_0}p(s,x,y)\Gamma(P_{T-s}f)(y)\right].
  \end{align*}
Applying (\ref{eq:LMP16}) to the last term above and applying the pointwise curvature dimension conditions to the first two terms, we have
\begin{align*}
 F'(s) &\geq 2\euler^{-2Ks}\left[\frac 1NP_s(\Delta P_{T-s}f)^2(x)\right.\\
 &\quad\left.-(K_0+K)\sum_{y\in V_0}p(s,x,y)\left(\euler^{2K_0(T-s)}P_{T-s}\Gamma f(y) +\frac{1-\euler^{2K_0(T-s)}}{K_0N}(\Delta P_{T-s}f)^2(y)\right)\right]\\
 &= -2(K+K_0)\euler^{2K_0T-2(K+K_0)s}\sum_{y\in V_0}P_{T-s}(\Gamma f)(y)p(s,x,y)\\
 &\quad +\frac{2\euler^{-2Ks}}{N}\left(P_s(\Delta P_{T-s}f)^2(x)+(K+K_0)\sum_{y\in V_0}p(s,x,y)\frac{\euler^{2K_0(T-s)}-1}{K_0}(\Delta P_{T-s}f)^2(y)\right).
\end{align*}
Jensen's inequality gives $$P_s(\Delta P_{T-s}f)^2(x)\geq (P_s\Delta P_{T-s}f)^2(x)=(\Delta P_Tf)^2(x).$$ Hence we have
\begin{align*}
 F'(s)
 &\geq -2(K+K_0)\euler^{2K_0T-2(K+K_0)s}\sum_{y\in V_0}P_{T-s}(\Gamma f)(y)p(s,x,y)+\frac{2\euler^{-2Ks}}{N}(\Delta P_Tf)^2(x)\\
 &\geq -2(K+K_0)\euler^{2K_0T-2(K+K_0)s}\Vert\Gamma f\Vert_\infty P_s1_{V_0}(x)+\frac{2\euler^{-2Ks}}{N}(\Delta P_Tf)^2(x).
\end{align*}

Therefore, we have
\begin{align*}
 F(T)-F(0)&=\euler^{-2KT}P_T(\Gamma f)(x)-\Gamma(P_Tf)(x)\\
 &=\int_0^TF'(s)\drm s \\
 &\geq -2(K+K_0)\euler^{2K_0T}\Vert\Gamma f\Vert_\infty\int_0^T\euler^{-2(K+K_0)s} P_s1_{V_0}(x)\drm s\\
 &\quad +\int_0^T\frac{2\euler^{-2Ks}}{N}\drm s(\Delta P_Tf)^2(x).
\end{align*}
Rearranging yields the claim.
\end{proof}
To control the distance to the negatively curved part $V_0$, we need to estimate $P_t 1_{V_0} (x)$ in terms of $\rho(x,V_0)$. This is given by the following theorem.

\begin{theorem}\label{thm:Pt1}
Let $G=(V,w,m)$ be a weighted, complete graph with non-degenerate vertex measure satisfying $CD(-K_0,N)$ for some $K_0 \geq 0$ and some $N>0$, let $\rho$ be an intrinsic metric, $x \in V$, and $W \subset V$ with $\rho(x,W) \geq R$. Then,
\begin{align}
P_t 1_W (x) \leq  \frac{\sqrt N}{R} \left(t\sqrt {K_0} + \sqrt{2t} \right) \label{eq:Pt1W leq}.
\end{align}

\end{theorem}
\begin{proof}
From (\ref{eq:LMP16}), we have for any bounded function $g$ with bounded $\Gamma g$,
$$
\frac{e^{2{K_0}s} - 1}{{K_0}N}\left(\Delta P_s g \right) ^2 \leq e^{2{K_0}s} P_s \Gamma g.
$$
Hence,
\begin{align}
\left|\Delta P_s g \right| &\leq \sqrt{\frac{{K_0}N}{1 - e^{-2{K_0}s}}} \sqrt{P_s\Gamma g} \leq \sqrt{\frac{{K_0}N}{1 - \frac{1}{1+2{K_0}s}}} \sqrt{P_s\Gamma g} = \sqrt{{K_0} + \frac{1}{2s}} \sqrt{N P_s\Gamma g} \nonumber\\&\leq \left(\sqrt {K_0} + \frac{1}{\sqrt{2s}} \right)\sqrt{N P_s\Gamma g}. \label{eq: LPsg leq PsGammag}
\end{align}
Let $g:= \left(1- \frac{\rho(x,\cdot)}{R} \right)_+$.
Then, $\Gamma g \leq 1/R^2$ and $g + 1_W \leq 1$ and thus by (\ref{eq: LPsg leq PsGammag}),
\begin{align*}
P_t 1_W (x)
&\leq 1-P_t g(x)
\leq \int_0^t |\Delta P_s g(x)|   ds \\
&\leq 	\int_0^t \left(\sqrt {K_0} +\frac{1}{\sqrt{2s}} \right)\sqrt{N P_s\Gamma g(x)} ds \\&
\leq \frac{\sqrt N}{R} \left(t\sqrt {K_0} + \sqrt{2t} \right).
\end{align*}
%
This finishes the proof.
\end{proof}

We show that a Bonnet-Myers type diameter bound still holds if one allows some negative curvature. In contrast to Bonnet-Myers, we will bound the distance to the negatively curved part $V_0$ of the graph from above, which proves part (iv) of Theorem \ref{thm:main}.

\begin{theorem}\label{thm:TubeNegativeCurvatureFiniteDimension}
	Let $G=(V,w,m)$ be a connected graph with non-degenerate vertex measure $m$, $K,K_0 >0$. Let $\emptyset \neq V_0 \subset V$.
	Suppose $G$ satisfies $$CD(K,N,x)\quad\forall \, x\in V \setminus V_0\quad\text{and}\quad CD(-K_0,N,x)\quad \forall \,x\in V_0.$$
	Let $\rho$ be an intrinsic metric with finite jump size $R_\rho>0$. Suppose $G$ is complete.
	Then for all $x_0 \in V$, one has
	\begin{align*}
	\rho(x_0,V_0) \leq R_\rho + 18.2e^{2K_0/K} \sqrt{\frac{N}{K_0+K}}.
	\end{align*}
	
\end{theorem}
\begin{remark}
	In case of bounded vertex degree, we can drop the non-degeneracy assumption of $m$.
\end{remark}

\begin{proof}
	
%
	
	By Theorem \ref{thm:Pt1}, we have $P_s1_{V_0} \leq  \frac{\sqrt N}{\rho(V_0,\cdot)} \left(s\sqrt {K_0} + \sqrt{2s} \right)$.
	Thus, Proposition \ref{gradientsg} implies that we have for any bounded function $f: V\to \RR$ with bounded $\Gamma f$,
%
	\begin{align}
	e^{-2KT} P_T \Gamma f - \Gamma P_T f
	&\geq  (\Delta P_T f)^2 \frac{1-e^{-2KT}}{KN} -  \frac {\sqrt{N}}{\rho(V_0,\cdot)} H(K,K_0,T)\Vert\Gamma f\Vert_\infty \label{eq:integrate F'}
	\end{align}
	with
	$$
	H(K,K_0,T) :=  2(K_0+K)e^{2K_0 T} \int_0^T
	e^{-2(K_0+K)s}\left(s\sqrt {K_0} + \sqrt{2s} \right) ds.
	$$

	
	On $V \setminus T^\rho_R(V_0)
	$, we have $\rho(V_0,\cdot) \geq R$.
	Activating  (\ref{eq:integrate F'}) towards $|\Delta P_T f|=|\partial_T P_T f|$ and throwing away the nonnegative term $\Gamma P_T f$ yields
	\begin{align}
	|\partial_T P_T f| \leq \sqrt{\left(e^{-2KT} + \frac{\sqrt{N}}{R}H(K,K_0,T)\right)\Vert\Gamma f\Vert_\infty} \sqrt{\frac{KN}{1-e^{-2KT}}}\label{eq: dtPtf leq sqrt}
	\end{align}	on $V\setminus T^\rho_R(V_0)$.
	On the other hand, activating  (\ref{eq:integrate F'}) towards $\Gamma P_T f$ and throwing away the nonnegative term $(\Delta P_T f)^2 \frac{1-e^{-2KT}}{Kn}$ yields
	\begin{align}
	\Gamma P_T f \leq \left({e^{-2KT} + \frac{\sqrt{N}}{R}H(K,K_0,T)}\right)\Vert\Gamma f\Vert_\infty \label{eq: GammaPTf leq Gamma plus H}
	\end{align}
	on $V\setminus T^\rho_R(V_0)$.
	
	This gives good control on the time derivative and gradient of the semigroup.

	We estimate
	\begin{align}
	H(K,K_0,T)&\leq  2(K_0+K)e^{2K_0 T} \int_0^\infty e^{-2(K_0+K)s}\left(s\sqrt {K_0} + \sqrt{2s} \right) ds\notag
	\\&
	=2(K_0+K)e^{2K_0 T}\cdot \frac 1 4  \left(\frac{\sqrt{K_0}}{(K_0+K)^2} + \frac{\sqrt{\pi}}{(K_0+K)^{3/2}} \right)\notag\\
	&= e^{2K_0 T}\cdot \frac 1 2  \left(\frac{\sqrt{K_0}}{K_0+K} + \sqrt{\frac{\pi}{K_0+K}} \right).\label{eq:Hesti}
	\end{align}
	Moreover, for $t<T$, one has
	\begin{align}
	H(K,K_0,t) &\leq 2(K_0+K) e^{2K_0 t} \int_0^T e^{-2(K_0+K)s} \left(s\sqrt {K_0} + \sqrt{2s} \right) ds \nonumber\\& = e^{2K_0(t-T)}H(K,K_0,T) \label{eq: Ht HT}.
	\end{align}
	
	By assumption, one has $\rho(x,y)\leq R_\rho$ whenever $x \sim y$.
	
	We fix $T, R,r>0$ and $x_0 \in V$.
	We suppose $\rho(x_0,V_0)= R+R_\rho + r$. Our aim is to show that $$\rho(V_0,x_0) \leq R_\rho + 18.2 e^{2K_0/K} \sqrt{\frac{N}{K_0+K}}. $$
	
	Let us explain the strategy of the remaining proof first. We will consider functions $f$ with $\Gamma f \leq 1$ and being constant outside of $B_r(x_0)$. We need the additional distance $R$ to have reasonable estimates for $\partial P_t f$ and $\Gamma P_t f$ for all vertices in $B_{r+R_\rho}(x_0)$. The $R_\rho$ is needed to separate $B_r(x_0)$ and $T^\rho_R(V_0)$, i.e., to guarantee that there are no edges connecting two vertices from the two sets respectively.

	The distance $R$ will be chosen later to ensure that the term $\frac {\sqrt{N}}{R} H(K,K_0,T)$ is small enough to obtain good estimates for $r$.

	Let us denote
	$$
	\mathcal F:= \{f \in C(V):f(x_0)=0\mbox{ and } \Gamma f \leq 1 \mbox{ and } f(y) = \sup f<\infty,\,\, \forall\, y \in V\setminus B_r(x_0)\}.
	$$
	We write $C_{\max}:=\sup\{f(y):f \in \mathcal F, y \in V\}< \infty$ due to connectedness.

	The idea is to take $f\in \mathcal F$ such that $\sup f$ is (close to be) maximal.
	Then, we take $P_t f$ and cut it off appropriately such that its cut-off version $h $ belongs to $\mathcal F$.
	By the estimate (\ref{eq: dtPtf leq sqrt}) for $|\partial_t P_t f|$, we can upper bound $|P_t f - f|$ outside the tube $T^\rho_R(V_0)$. On the other hand for $y_0 \notin B_r(x_0)$, we can upper bound
$|P_t f(x_0) - P_t f (y_0)|$ by the estimate (\ref{eq: GammaPTf leq Gamma plus H}) for $\Gamma P_t f$.
	By triangle inequality, we can thus upper bound (the cut-off version of)
$$|f(y_0)-P_tf(y_0)+P_tf(y_0)-P_tf(x_0)+P_tf(x_0)-f(x_0)|=|f(y_0)-f(x_0)|.$$
Notice that $ |f(y_0)-f(x_0)|\approx C_{\max} \geq r$, this leads to an upper estimate for $r$ when choosing $R$ and $T$ appropriately.

	The reason, why we take $C_{\max}$ as a substitute for the distance $\rho$, is that we need to forth- and back estimate between the distance and the gradient. The problem is that for $\rho$, we do not always have $f(x)- f(y) \leq \rho(x,y)$ when only assuming $\Gamma f \leq 1$.
	To avoid this problem, we take a certain resistance metric  between $x_0$ and $V \setminus B_r(x_0)$ given by $C_{\max}$.

	We now give the details.
	Let $\varepsilon>0$.
	We choose
	$f \in \mathcal F$ such that $C:=\sup f \geq C_{\max} - \varepsilon$.
	W.l.o.g., we can assume that $f \geq 0$.
	We have $C_{\max} \geq r$ since the function $\tilde{f}:=\min(\rho(x_0,\cdot),r) \in \mathcal F$ and $\sup \tilde{f} = r$ .



	Now, we set $g_{\max} :=\inf \{P_T f(y):y \in B_{r+R_\rho}(x_0) \setminus B_r(x_0)\}$ and
	$$g(x):= \begin{cases}
	P_T f(x) \wedge g_{\max}  &:x \in B_r(x_0) \\
	g_{\max} &: else
	\end{cases},
	$$where $a\wedge b:=\min\{a,b\}$ for $a,b\in \RR$.
	We remark that $B_{r+R_\rho}(x_0) \setminus B_r(x_0) \neq \emptyset$ due to the $R_\rho$ assumption (that is, $\rho(x,y)\leq R_\rho$ for $x\sim y$) and since $V_0 \neq \emptyset$ and since there is a path from $x_0$ to $V_0$ due to connectedness.
	The reason why we take the infimum over $B_{r+R_\rho}(x_0) \setminus B_r(x_0)$ and not over $V \setminus B_r(x_0)$ is that we want to control
	$|P_T f(y_0) - f(y_0)|$ at $y_0$ where the infimum of $P_Tf(\cdot)$ is almost attained. But this only works if $y_0$ is far away from the negatively curved set $V_0$.

In fact, we have 
\begin{equation}\label{eq: g}
(g(y)-g(z))^2\leq (P_Tf(y)-P_Tf(z))^2,\,\,\,\,\text{for all }\,\,y,z \in V.
\end{equation}
We can check (\ref{eq: g}) as follows. When $y,z\in B_r(x_0)$, we have
$$(g(y)-g(z))^2= (P_Tf(y)\wedge g_{\max}-P_Tf(z)\wedge g_{\max})^2\leq (P_Tf(y)-P_Tf(z))^2.$$
When one of the two vertices $y$ and $z$ lies in $B_r(x_0)$ and the other one lies outside $B_r(x_0)$, say $y\in B_r(x_0)$ and $z\in V\setminus B_r(x_0)$, we have
$$(g(y)-g(z))^2=(P_Tf(y)\wedge g_{\max}-g_{\max})^2.$$
In case that $P_Tf(y)\leq g_{\max}$, we have by the definition of $g_{\max}$ that $$(P_Tf(y)\wedge g_{\max}-g_{\max})^2=(P_Tf(y)-g_{\max})^2\leq (P_Tf(y)-P_Tf(z))^2.$$
Otherwise when $P_Tf(y)>g_{\max}$, we have $$(P_Tf(y)\wedge g_{\max}-g_{\max})^2=0\leq (P_Tf(y)-P_Tf(z))^2.$$
When $y,z\in V\setminus B_r(x_0)$, we have 
$$(g(y)-g(z))^2=(g_{\max}-g_{\max})^2=0\leq (P_Tf(y)-P_Tf(z))^2.$$
This finishes the verification of (\ref{eq: g}).

We obtain directly from (\ref{eq: g}) that 
$$\Gamma g\leq \Gamma P_Tf.$$
We observe that for all $y \in V \setminus B_{r+R_\rho}(x_0)$, there is no neighbor of $y$ in $B_r(x_0)$ due to the $R_\rho$ assumption. Since $g$ is constant on  $V \setminus
B_r(x_0)$, we obtain $\Gamma g = 0$ on $V \setminus B_{r+R_\rho}(x_0)$. 

Using $(\ref{eq: GammaPTf leq Gamma plus H})$, we derive further that
	$$\Gamma g \leq  {e^{-2KT}  + \frac{\sqrt{N}}{R}H(K,K_0,T)},$$
where we used the property $\Gamma f\leq 1$. We will later choose $T$ and $R$ such that this bound of $\Gamma g$ is significantly smaller than one.

Setting the function $h: V\to \RR$ to be
	$$
	h:= g \bigg/\sqrt{e^{-2KT} + \frac{\sqrt{N}}{R}H(K,K_0,T)},
	$$
	we have $\Gamma h\leq 1$. Therefore $h - h(x_0) \in \mathcal F$ and thus, $\sup h - h(x_0) \leq C_{\max}$. Hence,
	\begin{align}
	g_{\max}-g(x_0) \leq C_{\max} \sqrt{e^{-2KT} + \frac{\sqrt{N}}{R}H(K,K_0,T)}. \label{eq: gy0 minus gx0}
	\end{align}
Let $y_0$ be a vertex in $B_{r+R_\rho}(x_0) \setminus B_r(x_0)$ such that $P_T f(y_0)-g(y_0) < \varepsilon$.
	We obtain by (\ref{eq: dtPtf leq sqrt})
	\begin{align}
	|f(y_0) - g(y_0)| &\leq |f(y_0) - P_T f(y_0)| + \varepsilon \leq \varepsilon + \int_0^T |\partial_t P_t f(y_0)| dt \nonumber \\
	&\leq   \varepsilon + \int_0^T \sqrt{e^{-2Kt}  + \frac{\sqrt{N}}{R}H(K,K_0,t)} \sqrt{\frac{KN}{1-e^{-2Kt}}} dt. \label{eq:fy0 minus gy0}
	\end{align}
	Analogously, we have
	\begin{align}
	|g(x_0) - f(x_0)| \leq  \int_0^T \sqrt{e^{-2Kt}  + \frac{\sqrt{N}}{R}H(K,K_0,t)} \sqrt{\frac{KN}{1-e^{-2Kt}}} dt. \label{eq:fx0 minus gx0}
	\end{align}
	Noticing that $f(y_0)=\sup f=C$ and putting together (\ref{eq: gy0 minus gx0}), (\ref{eq:fy0 minus gy0}) and (\ref{eq:fx0 minus gx0}) yield
	\begin{align*}
	C_{\max}-\varepsilon &\leq C = f(y_0)-f(x_0) \\&\leq |f(y_0) - g(y_0)| + |g(y_0)-g(x_0)| + |g(x_0)-f(x_0)| \\
	& \leq   \varepsilon + 2\int_0^T \sqrt{e^{-2Kt}  + \frac{\sqrt{N}}{R}H(K,K_0,t)} \sqrt{\frac{KN}{1-e^{-2Kt}}} dt   \\
	& \qquad +  C_{\max}  \sqrt{e^{-2KT} + \frac{\sqrt{N}}{R}H(K,K_0,T)}.
	\end{align*}
	Letting $\varepsilon$ tend to zero yields
	%
	\begin{equation}
	r \leq C_{\max} \leq \frac{2\int_0^T \sqrt{e^{-2Kt}  + \frac{\sqrt{N}}{R}H(K,K_0,t)} \sqrt{\frac{KN}{1-e^{-2Kt}}} dt}{1-\sqrt{e^{-2KT} + \frac{\sqrt{N}}{R}H(K,K_0,T)} } \label{eq: r upper estimate}
	\end{equation}
	whenever the denominator is positive.
	
	We set $T:=1/K$ and $R:=4\sqrt{N} H(K,K_0,1/K)$. Then the denominator of the RHS of (\ref{eq: r upper estimate}) is $1-\sqrt{e^{-2}+1/4}>0$. Observe that (\ref{eq:Hesti}) implies
	\begin{align}
	R\leq    2 \sqrt N e^{2K_0/K}  \left(\frac{\sqrt{K_0}}{K_0+K} + \sqrt{\frac{\pi}{K_0+K}} \right). \label{eq: R def and estimate}
	\end{align}
	Next, we estimate the numerator.
	By (\ref{eq: Ht HT}), we have for $t \leq T=1/K$ that
	\begin{align}
	\frac{\sqrt{N}}{R}H(K,K_0,t) = \frac{H(K,K_0,t)}{4H(K,K_0,T)} \leq \frac 1 4 \exp \left[-2K_0\left(\frac 1 K - t \right) \right] \leq \frac 1 4. \label{eq: nRH leq 1 over 4}
	\end{align}
Therefore, we obtain	
	\begin{align*}
	&\int_0^T \sqrt{e^{-2Kt}  + \frac{\sqrt{N}}{R}H(K,K_0,t)} \sqrt{\frac{KN}{1-e^{-2Kt}}} dt \\
	\leq & \int_0^T \left[\sqrt{e^{-2Kt}}  + \sqrt{\frac{\sqrt{N}}{R}H(K,K_0,t)} \right] \sqrt{\frac{KN}{1-e^{-2Kt}}} dt \\
	\leq &\int_0^\infty \sqrt{e^{-2Kt}}\sqrt{\frac{KN}{1-e^{-2Kt}}} dt  + \int_0^T \frac 1 2 \sqrt{\frac{KN}{1-e^{-2Kt}}} dt\\
	=& \frac \pi 2 \sqrt{\frac{N}{K}} +  \int_0^T \frac 1 2\sqrt{\frac{KN}{1-e^{-2Kt}}} dt\\
       =& \frac \pi 2 \sqrt{\frac{N}{K}} +\frac{1}{2} \sqrt{\frac{N}{K}} \int_0^1 \frac{d\tau}{\sqrt{1-e^{-2\tau}}}\\
	\leq&  \frac {\pi+\arctanh(\sqrt{1-e^{-2}})} 2   \sqrt{\frac{N}{K}}.
	\end{align*}
	Thus, (\ref{eq: r upper estimate}) implies that
	\begin{align*}
	r  \leq \frac{ (\pi+\arctanh(\sqrt{1-e^{-2}}))}{1 - \sqrt{e^{-2} + \frac 1 4}} \sqrt{\frac{N}{K}}.
	\end{align*}

Using  this and (\ref{eq: R def and estimate}) yields
	\begin{align}\label{eq: rho first estimate}
	\rho(x_0,V_0) =r + R_\rho + R \hspace*{0.7\textwidth} \nonumber & \\ \leq  \frac{ (\pi+\arctanh(\sqrt{1-e^{-2}}))}{1 - \sqrt{e^{-2} + \frac 1 4}} \sqrt{\frac{N}{K}} + R_\rho +   2 \sqrt N e^{2K_0/K}  \left(\frac{\sqrt{K_0}}{K_0+K} + \sqrt{\frac{\pi}{K_0+K}} \right). &
	\end{align}
	
	It is left to show that $e^{2K_0/K}\sqrt{\frac{N}{K_0 +K}}$ is the dominating term in the sum and to give the corresponding coefficient.
	
	We start with comparing the addends in the brackets of (\ref{eq: rho first estimate}): we have
	$$
	\frac{\sqrt{K_0}}{K_0+K} \leq \frac 1 {\sqrt{K_0+K}},
	$$
	and, hence,
	$$
	R \leq 2(1+\sqrt{\pi})e^{2K_0/K}\sqrt{\frac{N}{K_0+K}}.
	$$
	Notice that one has for $s\geq 0$,
	$$
	e^{2s} \geq \sqrt{1+s}.
	$$
Thus via $s:=K_0/K$, we obtain
	$$
	e^{2K_0/K} \sqrt{\frac{N}{K_0+K}} = \sqrt{\frac{N}{K}} \cdot e^{2K_0/K} \sqrt{\frac{1}{1+K_0/K}} \geq \sqrt{\frac{N}{K}}.
	$$
	Hence,
	$$
	r \leq \frac{ (\pi+\arctanh(\sqrt{1-e^{-2}}))}{1 - \sqrt{e^{-2} + \frac 1 4}} e^{2K_0/K} \sqrt{\frac{N}{K_0+K}}.
	$$
	
	We infer that
	$$
	\rho(x_0,V_0) \leq R_\rho+r+R \leq R_\rho + 18.2e^{2K_0/K} \sqrt{\frac{N}{K_0+K}}.
	$$
	This finishes the proof.
\end{proof}

Combining Theorem~\ref{thm:TubeNegativeCurvatureFiniteDimension} and Proposition~\ref{prop:DistanceAndDimensionBoundedDegree}, we obtain a distance bound for bounded vertex degree and infinite dimension, what proves part (iii) of Theorem \ref{thm:main}.

\begin{corollary}
	Let $G=(V,w,m)$ be a graph with finite maximal vertex degree $\Deg_{\max}$, and let $K,K_0>0$.
	Let $\emptyset \neq V_0\subset V$ and suppose that $G$ satisfies $$CD(K,\infty,x)\quad\forall \, x\in V \setminus V_0\quad\text{and}\quad CD(-K_0,\infty,x)\quad \forall \,x\in V_0.$$ Then, for all $x \in V$, one has
	\begin{align*}
	d(x,V_0) \leq 1 + 26 e^{4K_0/K} \frac{\Deg_{\max}}{\sqrt{KK_0}}.
	\end{align*}
	
\end{corollary}

\begin{proof}
	Proposition~\ref{prop:DistanceAndDimensionBoundedDegree} yields $CD(-\widetilde{K_0}, N)$ on $V_0$ and $CD(\widetilde{K}, N)$ on $V\setminus V_0$ with
	\begin{align*}
	\widetilde{K_0} &= K_0+s, \\
	\widetilde{K} &= K-s,\\
	N&= \frac{2\Deg_{\max}}s.
	\end{align*}
	Applying Theorem~\ref{thm:TubeNegativeCurvatureFiniteDimension} with
	$\rho:=d \sqrt{\frac {2} {\Deg_{\max}}}$ which is intrinsic due to Proposition~\ref{prop:DistanceAndDimensionBoundedDegree} yields	
	\begin{align*}
	\rho(x,V_0) \leq R_\rho + 18.2 e^{2(K_0+s)/(K-s)} \sqrt{\frac{2\Deg_{\max}}{(K_0+K)s}}
	\end{align*}
	with $R_\rho = \sqrt{\frac {2} {\Deg_{\max}}}.$
	
	By choosing $s:= \frac{K K_0}{2(K+K_0)}$, we see
	\begin{align*}
	\frac{K_0+s}{K-s} = \frac{2K_0(K+K_0) + KK_0}{2K(K+K_0)-KK_0} \leq \frac{2K_0(2K+K_0)}{K(2K+K_0)} = \frac{2K_0}{K}
	\end{align*}
	and
	\begin{align*}
	\sqrt{\frac{2\Deg_{\max}}{(K_0+K)s}} = 2\sqrt{\frac{\Deg_{\max}}{KK_0}}.
	\end{align*}
	Therefore,
	\begin{align*}
	\rho(x,V_0) \leq R_\rho + 36.4 e^{4K_0/K} \sqrt{\frac{\Deg_{\max}}{KK_0}}
	\end{align*}
	which implies
	\begin{align*}
	d(x,V_0) \leq \rho(x,V_0)\sqrt{\frac{\Deg_{\max}}{2}} \leq 1 + 18.2 \sqrt{2} e^{4K_0/K} \frac{\Deg_{\max}}{\sqrt{KK_0}}\le 1+26 e^{4K_0/K} \frac{\Deg_{\max}}{\sqrt{KK_0}}.
	\end{align*}
	This finishes the proof.
\end{proof}

\section*{Acknowledgements}
C.~R. is grateful for the hospitality of Durham University, where parts of this work have been carried out during his visits. Moreover, the research was supported by the EPRSC Grant EP/K016687/1 \grqq{}Topology, Geometry and Laplacians of Simplicial Complexes\grqq{}. F.~M. wants to thank the German Research Foundation (DFG) and the German National Merit Foundation for financial support, and the Harvard University Center of Mathematical Sciences and Applications for their hospitality.

%
%

 \newcommand{\etalchar}[1]{$^{#1}$}

\noindent
Shiping Liu, \\
School of Mathematical Sciences, University of Science and Technology of China, Hefei 230026, Anhui Province, China\\
\texttt{spliu@ustc.edu.cn}\\
\\
Florentin M\"unch,\\
Max-Planck-Institut f\"ur Mathematik in den Naturwissenschaften, Leipzig, Inselstra{\ss}e 22, 04103 Leipzig, Germany\\
\texttt{Florentin.Muench@mis.mpg.de}\\
\\
Norbert Peyerimhoff, \\
Department of Mathematical Sciences, Durham University, Durham DH1, 3LE, United Kingdom\\
\texttt{norbert.peyerimhoff@durham.ac.uk}\\
\\
Christian Rose,\\
Max-Planck-Institut f\"ur Mathematik in den Naturwissenschaften, Leipzig, Inselstra{\ss}e 22, 04103 Leipzig, Germany\\
\texttt{Christian.Rose@mis.mpg.de}

\end{document}